\documentclass[12pt]{amsart}   
\usepackage{amsmath,amsthm,amssymb,marvosym}
\usepackage[english]{babel}
\usepackage[autostyle]{csquotes}
\usepackage{graphicx,subfig,wrapfig}
\usepackage{enumitem, ragged2e}
\usepackage{lipsum,hyperref}
\newtheorem{theorem}{Theorem}[section]
\newtheorem{Lemma}[theorem]{Lemma}
\newtheorem{Corollary}[theorem]{Corollary}
\newtheorem{Definition}[theorem]{Definition}
\newtheorem{Proposition}[theorem]{Proposition}
\newtheorem{Remark}[theorem]{Remark}
\newtheorem{Example}[theorem]{Example}

\begin{document} 
\title[Stability Theorems in Pointwise Dynamics]{Stability Theorems in Pointwise Dynamics}
\author[Abdul Gaffar Khan and Tarun Das]{Abdul Gaffar Khan$^{1}$ and Tarun Das$^{1}$}

\keywords{Expansive, Shadowing, Stability, Borel measure\vspace*{0.08cm}\\ 
\vspace*{0.01cm}
\Letter{Tarun Das} \\
\vspace*{0.08cm}
tarukd@gmail.com \\
\vspace*{0.01cm}
Abdul Gaffar Khan \\
\vspace*{0.08cm}
gaffarkhan18@gmail.com\\
$^{1}$\textit{Department of Mathematics, Faculty of Mathematical Sciences,} 
\textit{University of Delhi, Delhi, India.} 
}
\begin{abstract}
We introduce minimally expansive and GH-stable points for homeomorphisms on metric spaces and $\mu$-uniformly expansive, $\mu$-shadowable and strong $\mu$-topologically stable points for Borel measures (with respect to a homeomorphism on a metric space). 
We prove that: (i) minimally expansive shadowable point of a homeomorphism on a compact metric space is topologically stable and GH-stable. (ii) $\mu$-uniformly expansive $\mu$-shadowable point for a Borel measure $\mu$ (with respect to a homeomorphism on a compact metric space) is strong $\mu$-topologically stable. 
\end{abstract} 
\maketitle 

\section{Introduction} 
In the study of dynamical systems, the theory of shadowing property is a well established branch. The system with shadowing property forces a numerically computed orbit to follow an actual trajectory of the system. It plays a significant role in guaranteeing the positivity of topological entropy \cite[Theorem 3]{MI} and the topological stability \cite[Theorem 4]{WO} of a system. 

The notion of topological stability was first studied for Anosov diffeomorphisms by Walters in \cite{WA}. In \cite{WO}, Walters has extended this notion for the class of homeomorphisms and proved that expansive homeomorphisms with shadowing property on compact metric spaces are topologically stable. The important component in the hypothesis of Walters stability theorem called expansivity, were firstly studied by Utz in \cite{UU} for homeomorphisms on compact metric spaces. 
In \cite{RP}, Reddy started studying expansive behaviour of a map from local viewpoint and constructed a homeomorphism on a compact metric space which itself is not expansive but expansive at each point. Another stronger variant of an expansive point have also been introduced in \cite{ARP} with the name of uniformly expansive point. In \cite{ARP}, authors have proved that a homeomorphism on a compact metric space is expansive if and only if each point of a phase space is uniformly expansive. They have also proved that a homeomorphism on a compact space admitting non-periodic, shadowable \cite{MS}, uniformly positively expansive, non-wandering and non-isolated point has positive topological entropy \cite[Theorem 1]{ARP}. 
In \cite{DKDM}, authors have introduced expansive points for Borel measures and studied the connection of shadowable points with specification and Devaney chaotic points. 
In \cite{MS}, Morales has proved that a homeomorphism on a compact metric space has shadowing property if and only if each point of a phase space is shadowable. Some other variants of shadowable and expansive points have also been investigated in \cite{DKDM, KP, KQ}. 
Recently, Koo et. al. have studied the topological stability of a shadowable point. Precisely, they have proved that every shadowable point of an expansive homeomorphism on a compact metric space is topologically stable \cite[Corollary 3.16]{KLMP}. It is natural to find connection between the shadowing and the topological stability of an expansive point and its variants. 

In \cite{LMT}, Lee and Morales have introduced shadowing and topological stability for Borel measures with respect to a homeomorphism on a compact metric space and proved that every expansive measure with shadowing on a compact metric space is topologically stable. Recently Arbieto and Rojas in \cite{ART}, have introduced Gromov-Hausdorff distance to measure the distance between two maps on arbitrary phase spaces and then they have used it to introduce GH-stable systems. They have proved that an expansive homeomorphism on a compact metric space with shadowing property is GH-stable \cite[Theorem 4]{ART}.  

This paper is distributed as follows. In Section 2, we introduce minimally expansive points and discuss the dynamics of maps with uniformly expansive points and minimally expansive points. We provide sufficient conditions through which the type of expansivity of a point of a homeomorphism can be predicted by observing the behaviour of a sequence of functions uniformly converging to it. For the study of a relation between the global dynamics of a sequence of functions and its uniform limit, reader can refer to \cite{FDA, LA, SU}. We prove that every minimally expansive shadowable point of a homeomorphism on a compact metric space is topologically stable and GH-stable. In Section 3, we introduce $\mu$-uniformly expansive, $\mu$-shadowable and $\mu$-topologically stable points for Borel measures (with respect to a homeomorphism on a compact metric space). We prove that every $\mu$-uniformly expansive $\mu$-shadowable point for a Borel measure $\mu$ (with respect to a homeomorphism on a compact metric space) is strong $\mu$-topologically stable.

\section{\small{Minimally Expansive and GH-Stable Points}}
In this section, we introduce the notion of minimally expansive points of a homeomorphism on a metric space and  discuss the dynamics of maps having uniformly expansive points and minimally expansive points. Then, we introduce GH-stable points of a homeomorphism on a compact metric space and prove variants of stability theorems in pointwise setting. Firstly, we recall necessary notions required in this section.
\medskip

Throughout this paper, $(X, d^{X})$ and $(Y, d^{Y})$ denote compact metric spaces (if no confusion arises, then we write $d$ for the metric on $X$). For a given $\epsilon > 0$ and for each $x\in X$, we set $B(x, \epsilon) = \lbrace y\in X : d(x,y) < \epsilon \rbrace$ and $B[x, \epsilon] = \lbrace y\in X : d(x,y) \leq \epsilon \rbrace$.  
For a given self homeomorphism $f$ on $X$, the set $\mathcal{O}_{f}(x) = \lbrace f^{n}(x) : n\in \mathbb{Z}\rbrace$ denotes the orbit of $x\in X$, under $f$.  
\medskip

We say that a homeomorphism $f$ on $X$ is expansive on a subset $A\subset X$, if there exists a $\mathfrak{c} > 0$ (known as expansivity constant) such that for every pair of distinct points $x, y\in A$, there exists an $n\in \mathbb{Z}$ satisfying $d(f^{n}(x), f^{n}(y)) > \mathfrak{c}$. A homeomorphism $f$ on $X$ is said to be an expansive homeomorphism, if $f$ is expansive on $X$ \cite{UU}. A point $x\in X$ is said to be an expansive point of a homeomorphism $f$ on $X$, if there exists a $\mathfrak{c} > 0$ (known as expansivity constant of $f$ at $x$) such that for each $y\in X$ distinct from $x$, there exists an $n\in \mathbb{Z}$ satisfying $d(f^{n}(x), f^{n}(y)) > \mathfrak{c}$. A homeomorphism $f$ on $X$ is said to be a pointwise expansive homeomorphism, if every point in $X$ is an expansive point of $f$ \cite{RP}. A point $x\in X$ is said to be a uniformly expansive point of a homeomorphism $f$ on $X$, if there exists a $\mathfrak{c} > 0$ such that $f$ is expansive on $B(x, \mathfrak{c})$ with expansivity constant $\mathfrak{c}$ \cite{ARP}. The set of all uniformly expansive points of a homeomorphism $f$ on $X$ is denoted by $U_{f}(X)$. For every pair $x, y\in X$ and for each $\epsilon > 0$, we set $E_{y}(f, x, \epsilon) = \lbrace n\in \mathbb{Z} : d(f^{n}(x), f^{n}(y)) > \epsilon \rbrace$. 
\medskip

A sequence $\rho = \lbrace x_{n}\rbrace_{n\in \mathbb{Z}}$ in $X$  is said to be through a subset $B$ of $X$, if $x_{0}\in B$. 
Let $f$ be a homeomorphism on $X$ and let $\delta > 0$. A sequence $\rho = \lbrace x_{n}\rbrace_{n\in \mathbb{Z}}$ in $X$ is said to be a $\delta$-pseudo orbit for $f$ through $x$, if $x_{0} = x$ and $d(f(x_{n}), x_{n+1}) < \delta$, for each $n\in \mathbb{Z}$. A sequence $\rho$ is said to be $\delta$-traced through $f$, if there exists a $y\in X$ such that $d(f^{n}(y), x_{n}) < \delta$, for each $n\in \mathbb{Z}$. A point $x\in X$ is said to be a shadowable point of $f$, if for every $\epsilon > 0$, there exists a $\delta > 0$ such that every $\delta$-pseudo orbit for $f$ through $x$ can be $\epsilon$-traced through $f$ by some point of $X$ \cite{MS}. The set of all shadowable points of $f$ is denoted by $Sh(f)$.
\medskip

For a metric space $(X, d^{X})$ and $A,B\subset X$, we define:
\begin{align*}
d^{X}(A, B) = \inf \lbrace d^{X}(a, b) : (a, b)\in A\times B\rbrace
\end{align*}

We replace $A$ by $``a"$, whenever $A = \lbrace a\rbrace$. 
\medskip

The Hausdorff distance between $A$ and $B$ is given by: 
\begin{align*}
d_{H}^{X}(A, B) = \max \lbrace \sup_{a\in A} d^{X}(a, B), \sup_{b\in B}d^{X}(A, b)\rbrace
\end{align*}
%\medskip

An isometry between $(X, d^{X})$ and $(Y, d^{Y})$ is an onto map $i : (X, d^{X}) \rightarrow (Y, d^{Y})$ satisfying 
$d^{X}(x, x') = d^{Y} (i(x), i(x')) $, for every pair $x, x'\in X$. 
For a given $\delta > 0$, the map $i : X\rightarrow Y$ between $(X, d^{X})$ and $(Y, d^{Y})$ is said to be a $\delta$-isometry, if 
\begin{align*}
\max \lbrace d_{H}^{Y}(i(X), Y), \sup_{x, x'\in X}|d^{Y}(i(x), i(x')) - d^{X}(x, x')|\rbrace < \delta
\end{align*}

The $C^{0}$-distance between maps $f : (X, d^{X})\rightarrow (Y, d^{Y})$ and $\overline{f}: (X, d^{X})\rightarrow (Y, d^{Y})$ is given by:
\begin{align*}
d_{C^{0}}^{Y}(f, \overline{f}) = \sup_{x\in X}d^{Y}(f(x), \overline{f}(x))
\end{align*}

The $C^{0}$-Gromov-Hausdorff distance between maps $h: (X, d^{X})\rightarrow (X, d^{X})$ and $g: (Y, d^{Y})\rightarrow (Y, d^{Y})$ is given by:
\begin{center}
$d_{GH^{0}}(f, g) = \inf \lbrace \delta > 0:$ there exist $\delta$-isometries $i : X\rightarrow Y$ and $j : Y\rightarrow X$ such that $d_{C^{0}}^{Y}(g\circ i, i\circ h) < \delta$ and $d_{C^{0}}^{X}(j\circ g, h\circ j)  < \delta \rbrace$
\end{center}
\medskip

Let $f$ be a homeomorphism on $(X, d^{X})$ and $x\in X$. Then,
\begin{enumerate}
\item[(i)] $f$ is said to be topologically stable, if for every $\epsilon > 0$, there exists a $\delta > 0$ such that for every homeomorphism $g$ on $X$ satisfying $d_{C^{0}}(f, g) < \delta$, there exists a continuous map $h : X\rightarrow X$ such that $f\circ h = h\circ g$ and $d(h(x), x) < \epsilon$, for each $x\in X$ \cite{WO}.
\item[(ii)] $f$ is said to be GH-stable, if for every $\epsilon > 0$, there exists a $\delta > 0$ such that for every homeomorphism $g$ on $(Y, d^{Y})$ satisfying $d_{GH^{0}}(f, g) < \delta$, there exists a continuous $\epsilon$-isometry $h : Y\rightarrow X$ such that $f\circ h = h\circ g$ \cite{ART}.
\item[(iii)] $x$ is said to be a topologically stable point of $f$, if for every $\epsilon > 0$, there exists a $\delta > 0$ such that for every homeomorphism $g$ on $X$ satisfying $d_{C^{0}}(f, g)\leq \delta$, there exists a continuous map $h : \overline{\mathcal{O}_{g}(x)} \rightarrow X$ such that $f\circ h = h\circ g$ and $d(h(z), z) \leq \epsilon$, for each $z\in \overline{\mathcal{O}_{g}(x)}$ \cite{KLMP}. 
\end{enumerate}
\medskip

Now, we define minimally expansive points which plays an important role in guaranteeing the stability of a system at a point.
\begin{Definition}
Let $f$ be a homeomorphism on $X$. Then, a point $x\in X$ is said to be a minimally expansive point of $f$, if there exists a $\mathfrak{c} > 0$ such that for each $y\in B(x, \mathfrak{c})$, $f$ is expansive on $\overline{\mathcal{O}_{f}(y)}$ with expansivity constant $\mathfrak{c}$. Such a constant $\mathfrak{c}$ is said to be an expansivity constant for minimal expansivity of $f$ at $x$. The set of all minimally expansive points of $f$ is denoted by $M_{f}(X)$. 
\phantomsection\label{D3.1}
\end{Definition}

In \cite{RP}, Reddy did not point out any result which holds explicitly for an expansive homeomorphism but does not hold for a pointwise expansive homeomorphism. Since, a homeomorphism $f$ on $X$ is expansive if and only if $U_{f}(X) = X$ \cite[Proposition 6]{ARP}, nothing additional can be said for a homeomorphism $f$ with $U_{f}(X) = X$. On the contrary, from the following discussion we get that there are results which holds for every expansive homeomorphism but not for a homeomorphism $f$ satisfying $M_{f}(X) = X$. 
Clearly, if $f$ is expansive, then $M_{f}(X) = X$, but the converse need not be true. Infact, if $R_{\mathbb{S}^{1}}$ denotes the rational rotation on a unit circle equipped with the usual metric, then $R_{\mathbb{S}^{1}}$ is not expansive but $M_{R_{\mathbb{S}^{1}}}(\mathbb{S}^{1}) = \mathbb{S}^{1}$. Therefore, circle admits a homeomorphism under which every point is a minimally expansive point, but there does not exists any pointwise expansive homeomorphism of an circle \cite[Corollary]{RP}. Therefore, minimally expansive point need not be an expansive (uniformly expansive) point. This also implies that, there exists a homeomorphism $f$ with uncountable periodic points and satisfies $M_{f}(X) = X$, but \cite[Lemma]{RP} guarantees that, if $f$ is a pointwise expansive homeomorphism on a compact metric space, then the set of all periodic points of $f$ is countable. 
\medskip

We remark here that, if $f$ is a transitive homeomorphism on $X$, then either $M_{f}(X) = \phi$ or $f$ is expansive. Since, either a transitive homeomorphism have shadowing or have no shadowable points \cite[Theorem 1.2]{KQ} and every expansive homeomorphism with the shadowing property on a compact metric space is topologically stable \cite[Theorem 4]{WO} and GH-stable \cite[Theorem 4]{ART}, we can conclude that every transitive homeomorphism having atleast one minimally expansive point and atleast one shadowable point, is topologically stable and GH-stable.

\begin{Proposition}
Let $f$ and $g$ be homeomorphisms on $X$ and $Y$ respectively. If $h : X \rightarrow Y$ is a homeomorphism, then the following statements are true:
\begin{enumerate}
\item[(1)] $U_{f}(X) = U_{f^{-1}}(X)$.
\item[(2)] $U_{h\circ f \circ h^{-1}}(Y) = h(U_{f}(X))$.
\item[(3)] $U_{f^{k}}(X)$ is an invariant set of $f$, for each $k\in \mathbb{Z}$. In particular, the set of all uniformly expansive points is invariant under $f$.
\item[(4)] $U_{f\times g}(X\times Y) = U_{f}(X)\times U_{g}(Y)$.
\item[(5)] $M_{f}(X) = M_{f^{-1}}(X)$.
\item[(6)] $M_{h\circ f \circ h^{-1}}(Y) = h(M_{f}(X))$.
\item[(7)] $M_{f^{k}}(X)$ is an invariant set of $f$, for each $k\in \mathbb{Z}$. In particular, the set of all minimally expansive points is invariant under $f$.
\end{enumerate}
\label{T3.3}
\end{Proposition}
\begin{proof}
Suppose that $X\times Y$ is equipped with the maximum metric $D$ defined as $D((x_{1}, y_{1}), (x_{2}, y_{2})) = \max \lbrace d^{X}(x_{1},x_{2}), d^{Y}(y_{1}, y_{2})\rbrace\text{, for all }(x_{1}, y_{1}), (x_{2}, y_{2})\in X\times Y$. Proofs of (1), (4) and (5) follows from corresponding definitions, (3) follows from (2) and (7) follows from (6). Since proofs of (2) and (6) are similar, we only prove (6). 

\begin{enumerate}
\item[(6)] Let $h: X\rightarrow Y$ be a homeomorphism, $y\in h(M_{f}(X))$ and let $z\in M_{f}(X)$ with expansivity constant $\mathfrak{c}$ such that $y = h(z)$. Choose a $\mathfrak{d} > 0$ such that $d^{Y}(y_{1}, y_{2}) < \mathfrak{d}$ implies $d^{X}(h^{-1}(y_{1}), h^{-1}(y_{2}))$ $< \mathfrak{c}$, for all $y_{1}, y_{2}\in Y$. We claim that, $y\in M_{h\circ f\circ h^{-1}}(Y)$ with expansivity constant $\mathfrak{d}$. 
Now, fix $w\in B(y, \mathfrak{d})$ and choose a pair of distinct points $u, v\in \overline{\mathcal{O}_{h\circ f\circ h^{-1}}(w)}$. Set $w' = h^{-1}(w)\in B(z, \mathfrak{c})$. Choose sequence of integers $\lbrace n_{i}^{u}\rbrace_{i\in \mathbb{Z}}$, $\lbrace n_{i}^{v}\rbrace_{i\in \mathbb{Z}}$ satisfying $h\circ f^{n_{i}^{u}}\circ h^{-1}(w)\rightarrow u$ and $h\circ f^{n_{i}^{v}}\circ h^{-1}(w)\rightarrow v$. 
Therefore, $f^{n_{i}^{u}}\circ h^{-1}(w)\rightarrow h^{-1}(u)$ and $f^{n_{i}^{u}}\circ h^{-1}(w)\rightarrow h^{-1}(u)$ implying $h^{-1}(u), h^{-1}(v) \in \overline{\mathcal{O}_{f}(w')}$. Therefore, $d^{Y}(h\circ f^{n}\circ h^{-1}(u), h\circ f^{n}\circ h^{-1}(v)) < \mathfrak{d}$, for each $n\in \mathbb{Z}$ implying $d^{X}(f^{n}\circ h^{-1}(u), f^{n}\circ h^{-1}(v)) < \mathfrak{c}$, for each $n\in \mathbb{Z}$ which contradicts the expansivity of $f$ on $\overline{\mathcal{O}_{f}(w')}$. Since $w$ is chosen arbitrarily, we get that $y\in M_{h\circ f \circ h^{-1}}(Y)$ with expansivity constant $\mathfrak{d}$. Hence, $h(M_{f}(X))\subset M_{h\circ f\circ h^{-1}}(Y)$. Replace $h$ by $h^{-1}$ and $f$ by $h\circ f\circ h^{-1}$ in the last inclusion to complete the proof. 
\end{enumerate}
\end{proof}

Recall that, a sequence of homeomorphisms $\lbrace f_{n}\rbrace_{n\in \mathbb{N}}$ on $X$ is said to be uniformly convergent to a homeomorphism $f$ on $X$ or $f_{n}\xrightarrow[]{\text{uc}} f$, if for every $\epsilon > 0$, there exists an $N \in \mathbb{N}$ such that $d(f_{n}(x), f(x)) < \epsilon$, for all $n\geq N$ and for each $x\in X$. 
\medskip

\begin{Lemma}
Suppose that a sequence of homeomorphisms $\lbrace f_{n}\rbrace_{n\in \mathbb{N}}$ on $X$ is uniformly convergent to a homeomorphism $f$ on $X$. Then, for each triplet $\epsilon > 0$, $k\in \mathbb{N}$ and $x\in X$, there exists an $N = N(\epsilon, k, x) \in \mathbb{N}$ such that $d(f_{n}^{k}(x), f^{k}(x)) < \epsilon$, for all $n\geq N$.
\label{L3.4}
\end{Lemma}
\begin{proof}
Recall that, for a given $\epsilon > 0$ and for each positive integer $l$, there exists an $N = N(\epsilon, k) \in \mathbb{N}$ such that $d(f_{n}^{l}(x), f^{l}(x)) < \epsilon$, for all $n\geq N$ and for each $x\in X$ \cite[Lemma 3.1]{AAU}. Now, proof follows immediately.
\end{proof}

\begin{theorem} 
Let $\lbrace f_{n}\rbrace_{n\in \mathbb{N}}$ be a sequence of homeomorphisms on $X$ and let $f$ be a homeomorphism on $X$ such that $f_{n}\xrightarrow{uc} f$ and $f_{n}^{-1}\xrightarrow{uc} f^{-1}$. Then, the following statements are true:
\begin{enumerate}
\item[(1)] $x\in U_{f}(X)$ if and only if there exists a $\delta > 0$ such that for every pair of distinct points $y, z\in B(x, \delta)$, $\cup_{m\geq 1} \cap_{n\geq m} E_{z}(f_{n}, y, \delta)\neq \phi$.
\item[(2)] $x\in M_{f}(X)$ if and only if there exists a $ \delta >0 $ such that for each $y\in B(x, \delta)$, if for every pair of distinct points $u, v\in \overline{\mathcal{O}_{f}(y)}$, $\cup_{m\geq 1} \cap_{n\geq m} E_{v}(f_{n}, u, \delta)\neq \phi$.
\end{enumerate}
\label{T3.5}
\end{theorem}
\begin{proof}
Let $\lbrace f_{n}\rbrace_{n\in \mathbb{N}}$ be a sequence of homeomorphisms on $X$ and let $f$ be a homeomorphism on $X$ such that $f_{n}\xrightarrow{uc} f$ and $f_{n}^{-1}\xrightarrow{uc} f^{-1}$. Since proofs of (1) and (2) are similar, we prove only (1).
\begin{enumerate}
\item[(1)] Let $x\in U_{f}(X)$ with expansivity constant $ \mathfrak{c}$ and choose distinct points $y, z\in B(x, \frac{\mathfrak{c}}{3})$. Fix $k(y,z) = k \in \mathbb{Z}$ such that $d(f^{k}(y), f^{k}(z)) > \mathfrak{c}$. From Lemma \ref{L3.4}, choose $N\in \mathbb{N}$ such that $d(f_{n}^{k}(u), f^{k}(u)) < \frac{\mathfrak{c}}{3}$, for all $n\geq N$ and for each $u\in \lbrace y, z\rbrace$. 
Note that, $d(f_{n}^{k}(y), f_{n}^{k}(z)) > \frac{\mathfrak{c}}{3}$, for all $n\geq N$ implying $k\in \cap_{n\geq N} E_{x}(f_{n}, y, \frac{\mathfrak{c}}{3})$ i.e. $\cup_{m\geq 1} \cap_{n\geq m} E_{z}(f_{n}, y, \frac{\mathfrak{c}}{3})\neq \phi$. Since $y,z$ are chosen arbitrarily, we have $\cup_{m\geq 1} \cap_{n\geq m} E_{z}(f_{n}, y, \frac{\mathfrak{c}}{3})\neq \phi$, for every pair of distinct points $y, z\in B(x, \frac{\mathfrak{c}}{3})$. 
Conversely, choose a $ \delta > 0$ such that  $\cup_{m\geq 1} \cap_{n\geq m} E_{z}(f_{n},y , \delta)\neq \phi$, for every pair of distinct points $u, v\in B(x, \delta)$. For a pair of distinct points $y,z\in B(x, \frac{\delta}{3})$, choose  $M(y, z) = M\in \mathbb{N}$ such that $\cap_{n\geq M} E_{z}(f_{n}, y, \delta)$ $\neq \phi$. Fix $p\in \cap_{n\geq M} E_{z}(f_{n}, y, \delta)$. From Lemma \ref{L3.4}, choose $N\geq M$ such that $d(f_{n}^{p}(u), f^{p}(u)) < \frac{\delta}{3}$, for all $n\geq N$ and for each $u\in \lbrace y, z\rbrace$. Therefore, we have $d(f^{p}(y), f^{p}(z)) > \frac{\delta}{3}$. Since $\delta$ does not depends on $y$ and $z$, we get that $x\in U_{f}(X)$ with expansivity constant $\frac{\delta}{3}$.
\end{enumerate}
\end{proof}

\begin{Definition}
Let $f$ be a homeomorphism on $(X, d^{X})$. Then, a point $x\in X$ is said to be a GH-stable point of $f$, if for every $\epsilon > 0$, there exists a $\delta > 0$ such that for every homeomorphism $g$ on $(Y, d^{Y})$ satisfying $d_{GH^{0}}(f, g) < \delta$, there exists an $\epsilon$-isometry $j: Y\rightarrow X$ such that for each $y\in j^{-1}(x)$ there exists a continuous map $h : \overline{\mathcal{O}_{g}(y)} \rightarrow X$ satisfying $d^{X}(h(z), j(z)) < \epsilon$, for each $z\in \overline{\mathcal{O}_{g}(y)}$ and $f\circ h = h\circ g$. The set of all GH-stable points of $f$ is denoted by $GH_{f}(X)$.
\label{D5.1}
\end{Definition}

Clearly, if $f$ is GH-stable, then $GH_{f}(X) = X$. To prove the stability theorem in pointwise setting, we need the following Lemma.

\begin{Lemma} \label{L5.4} 
Let $f$ be a homeomorphism on $X$ and let $x\in X$ be a minimally expansive point of $f$ with expansivity constant $\mathfrak{c}$. Then, for each $y\in B(x, \mathfrak{c})$ and for each $0 < \epsilon < \mathfrak{c}$, there exists an $N(y, \epsilon)  = N \in \mathbb{N}$ such that for every pair $u, v \in \mathcal{O}_{f}(y)$ satisfying $d(f^{n}(u), f^{n}(v)) < \mathfrak{c}$, for all $-N\leq n\leq N$, we have $d(u, v) < \epsilon$. 
\end{Lemma}
\begin{proof}
Assume the contrary. Choose $0 < \epsilon < \mathfrak{c}$ such that for each $N\in \mathbb{N}$, there exists a pair $u_{N}, v_{N}\in \mathcal{O}_{f}(y)$ satisfying $d(f^{n}(u_{N}), f^{n}(v_{N})) < \mathfrak{c}$, for all $-N\leq n\leq N$ and $d(u_{N}, v_{N}) \geq \epsilon$. Since $X$ is compact, we can choose $u, v\in \overline{\mathcal{O}_{f}(y)}$ such that $u_{N}\rightarrow u$ and $v_{N}\rightarrow v$. Therefore, $u\neq v$ and $d(f^{n}(u), f^{n}(v)) \leq \mathfrak{c}$, for each $n\in \mathbb{Z}$ which contradicts the expansivity of $f$ on $\overline{\mathcal{O}_{f}(y)}$.
\end{proof}

\begin{theorem}
Let $f$ be a homeomorphism on $X$. If $x\in X$ is a minimally expansive shadowable point of $f$, then $x$ is a topologically stable and a GH-stable point of $f$.
\label{T5.7}
\end{theorem}
\begin{proof}
We only prove the latter case. Proof of first case can be done on similar lines. 
Let $x\in X$ be a minimally expansive shadowable point of $f$ with expansivity constant $\mathfrak{c} > 0$. For a given $\epsilon > 0$ and for $\eta = \frac{\min\lbrace \epsilon, \mathfrak{c}\rbrace}{16}$, choose $ 0 < \delta < \eta$ by shadowing of $f$ at $x$. Choose a homeomorphism $g$ on $Y$ satisfying $d_{GH^{0}}(f, g) < \delta$. Therefore, there exist $\delta$-isometries $i : X\rightarrow Y$ and $j : Y\rightarrow X$ such that $d_{C^{0}}^{Y}(g\circ i, i\circ f) < \delta$ and $d_{C^{0}}^{X}(j\circ g, f\circ j) < \delta$. We claim that, $j$ is a required $\epsilon$-isometry  to establish that $x$ is a GH-stable point of $f$. If $j^{-1}(x) = \phi$, then we are done. Otherwise, assume that $y\in j^{-1}(x)$ and consider a sequence $\lbrace x_{n}^{y} = j(g^{n}(y))\rbrace_{n\in \mathbb{Z}}$. Clearly,
\begin{align*}
d^{X}(x_{n+1}^{y}, f(x_{n}^{y})) &= d^{X}(j(g^{n+1}(y)), f(j(g^{n}(y)))) \\ 
&= d^{X}(j\circ g(g^{n}(y)), f\circ j(g^{n}(y))) \\ 
&< \delta \hspace*{0.2cm} \text{for} \hspace*{0.1cm} \text{each} \hspace*{0.1cm}n\in \mathbb{Z}
\end{align*}   
Hence, $\lbrace x_{n}^{y}\rbrace_{n\in \mathbb{Z}}$ forms a $\delta$-pseudo orbit for $f$ through $x$. By shadowing of $f$ at $x$, choose and fix $z\in B(x, \eta)$ satisfying $d^{X}(f^{n}(z), j(g^{n}(y))) < \eta$, for each $n\in \mathbb{Z}$. Define $h : \mathcal{O}_{g}(y) \rightarrow X$ by $h(g^{n}(y)) = f^{n}(z)$, for each $n\in \mathbb{Z}$. 
\vspace{0.2cm}

To check that, $h$ is well defined, choose $k, m\in \mathbb{Z}$ such that $g^{k}(y) = g^{m}(y)$. Then $j(g^{n+k}(y)) = j(g^{n+m}(y))$, for each $n\in \mathbb{Z}$ and hence
\begin{align*}
d^{X}(f^{n}(f^{k}(z)), f^{n}(f^{m}(z))) &\leq d^{X}(f^{n+k}(z), j(g^{n+k}(y))) + d^{X}(j(g^{n+k}(y)), j(g^{n+m}(y))) \\
&\hspace*{0.2cm}+  d^{X}(j(g^{n+m}(y)), f^{n+m}(z)) \\
&= d^{X}(f^{n+k}(z), j(g^{n+k}(y))) +  d^{X}(j(g^{n+m}(y)), f^{n+m}(z)) \\
&\leq 2\eta < \mathfrak{c} \hspace*{0.2cm} \text{for} \hspace*{0.1cm} \text{each} \hspace*{0.1cm} n\in \mathbb{Z}
\end{align*}

Since, $x$ is a minimally expansive point of $f$ with expansivity constant $\mathfrak{c}$ and $z\in B(x, \mathfrak{c})$, $f$ is expansive on $\overline{\mathcal{O}_{f}(z)}$ with expansivity constant $\mathfrak{c}$ and hence $f^{k}(z) = f^{m}(z)$. Therefore, $h$ is well defined. Moreover,
\begin{align*}
(f\circ h)(g^{n}(y)) &= f\circ (f^{n}(z)) = f^{n+1}(z) \\
&= h(g^{n+1}(y)) = h(g(g^{n}(y))) \\
&= (h\circ g )(g^{n}(y)) \hspace*{0.2cm} \text{for} \hspace*{0.1cm} \text{each} \hspace*{0.1cm} n\in \mathbb{Z}
\end{align*}
Therefore, $f\circ h(u) = h\circ g(u)$, for each $u\in \mathcal{O}_{g}(y)$. Also, $d^{X}(h(g^{n}(y)), j(g^{n}(y))) < \eta$, for each $n\in \mathbb{Z}$ implying $d^{X}(h(u), j(u)) < \eta$, for each $u\in \mathcal{O}_{g}(y)$.
\vspace{0.2cm}
 
Now, we claim that $h$ is uniformly continuous. For $z$ as above and $0 < \epsilon < \mathfrak{c}$, choose an $N\in \mathbb{N}$ from Lemma \ref{L5.4}. By uniform continuity of $g$, choose $0 < \gamma < \epsilon$ such that $d^{Y}(u, v) < \gamma$ implies $d^{Y}(g^{n}(u), g^{n}(v)) < \frac{\mathfrak{c}}{2}$, for all $-N\leq n \leq N$ and for all $u, v\in Y$. Therefore, for all $u, v\in \mathcal{O}_{g}(y)$ satisfying $d^{Y}(u, v) < \gamma$, we have
\begin{align*}
d^{X}(f^n(h(u)), f^n(h(v))) &= d^{X}(h(g^n(u)), h(g^n(v)))\\
&\leq d^{X}(h(g^n(u)), j(g^n(u))) + d^{X}(j(g^n(u)), j(g^n(v))) \\
&\hspace*{0.2cm} + d^{X}(h(g^n(v)), j(g^n(v)))\\
&\leq d^{X}(h(g^n(u)), j(g^n(u))) + \delta + d^{Y}(g^n(u), g^n(v)) \\
&\hspace*{0.2cm} + d^{X}(h(g^n(v)), j(g^n(v)))\\
&\leq 3\eta + \frac{\mathfrak{c}}{2} \\
&< \mathfrak{c} \hspace*{0.2cm} \text{for} \hspace*{0.2cm} \text{all} \hspace*{0.2cm} -N\leq  n \leq N
\end{align*}
Therefore, $d^{X}(h(u), h(v)) < \epsilon$ implying $h$ is uniformly continuous. Since, $Y$ is compact and $d^{X}(j(y_{1}), j(y_{2}))$ $< \delta + d^{Y}(y_{1}, y_{2})$, for all $y_{1}, y_{2}\in Y$, we can extend $h$ continuously to a  function $H: \overline{\mathcal{O}_{g}(y)} \rightarrow X$ such that $f\circ H = H\circ  g$ and $d^{X}(H(u), j(u)) < \epsilon$, for each $u\in \overline{\mathcal{O}_{g}(y)}$, which completes the proof.
\end{proof}

\begin{Corollary}
Every minimally expansive point of a homeomorphism on a compact manifold of dimension atleast $2$ is shadowable if and only if it is topologically stable.
\label{C5.6}
\end{Corollary}
\begin{proof}
Recall that, every topologically stable point of a homeomorphism on a compact manifold of dimension atleast $2$ is shadowable \cite[Lemma 3.11]{KLMP}. Now, use Theorem \ref{T5.7} to complete the proof.
\end{proof}

\begin{Corollary}
Let $f$ be a homeomorphism on $X$. Then, the following statements are true:
\begin{enumerate}
\item[(1)] Every isolated fixed point under $f$ is topologically stable and GH-stable.
\item[(2)] If $f$ is expansive, then every shadowable point of $f$ is topologically stable and GH-stable.
\end{enumerate}
\label{C4.8}
\end{Corollary}

\begin{Corollary}
Let $I_{X}$ be the identity map on a totally disconnected space $X$. Then, every point of $X$ is a topologically stable and a GH-Stable point of $I_{X}$.
\label{C4.9}
\end{Corollary}

\begin{Example}
Let $g$ be an expansive homeomorphism with shadowing property on an uncountable compact metric space $(Y,d_0)$. Let $p$ be a periodic point of $g$ with prime period $t\geq 2$. Let $X=Y\cup E$, where $E$ is an infinite enumerable set. Set $Q=\bigcup_{k\in\mathbb{N}} \lbrace 1,2,3\rbrace\times\lbrace k\rbrace\times\lbrace 0,1,2,3,...,t-1\rbrace$. Suppose that $r:\mathbb{N}\rightarrow E$ and $s:Q\rightarrow \mathbb{N}$ are bijections. Then, consider the bijection $q:Q\rightarrow E$ defined as $q(i,k,j)=r(s(i,k,j))$, for each $(i,k,j)\in Q$. Therefore, any point $x\in E$ has the form $x=q(i,k,j)$ for some $(i,k,j)\in Q$. 
\medskip

Define a function $d:X\times X\rightarrow\mathbb{R}^+$ by  
\[d(a,b)=\begin{cases} 
0 & \textnormal {if $a=b$},
\\
d_0(a,b) & \textnormal {if $a,b\in Y$} 
\\
\frac{1}{k}+d_0(g^j(p),b) & \textnormal {if $a=q(i,k,j)$ and $b\in Y$}
\\
\frac{1}{k}+d_0(a,g^j(p)) & \textnormal {if $a\in Y$ and $b=q(i,k,j)$} 
\\
\frac{1}{k} & \textnormal {if $a=q(i,k,j)$,$b=q(l,k,j)$ and $i\neq l$} 
\\
\frac{1}{k}+\frac{1}{m}+d_0(g^j(p),g^r(p)) & \textnormal {if $a=q(i,k,j)$,$b=q(i,m,r)$ and $k\neq m$ or $j\neq r$}    
\end{cases}\] 
Define a map $f:X\rightarrow X$ by 
\begin{center}
\[f(x)=\begin{cases}
g(x) & \textnormal {if $x\in Y$}
\\
q(i,k,(j+1))$ mod $t & \textnormal {if $x=q(i,k,j)$} 
\end{cases}\] 
\end{center}

Following steps as in \cite[(4) on Page 3742-3743]{CCN} and from \cite[Example 3.13]{DKDM}, we get that $(X,d)$ is a compact metric space, $f$ is a homeomorphism with shadowing property and $f$ is not pointwise expansive. Let $\mathfrak{c}$ be an expansivity constant of $g$.  We claim that $M_{f}(X) = X$. 
Since, each $a\in Q$ is an isolated periodic point, we get that $a\in M_{f}(X)$. Note that, each $x\in Y\setminus \mathcal{O}_{g}(p)$ is a minimally expansive point of $f$ with expansivity constant $0 < \mathfrak{d} < \min_{0\leq j < t}\lbrace d_{0}(g^{j}(p), x), \mathfrak{c}\rbrace$. Also, $x\in \mathcal{O}_{g}(p)$ is minimally expansive with expansivity constant $0 < \mathfrak{d} < \min_{0\leq j, r < t}\lbrace d_{0}(g^{j}(p), g^{r}(p)), \mathfrak{c}\rbrace$, where $j\neq r$. From Theorem \ref{T5.7}, every point of $X$ is a topologically stable and a GH-Stable point of $f$. 
\label{E5.12}
\end{Example}

\section{\small{$\mu$-Uniformly Expansive, $\mu$-Shadowable and $\mu$-Topologically Stable points}}
In this section, we introduce $\mu$-uniformly expansive and $\mu$-shadowable points for a Borel measure $\mu$ (with respect to a homeomorphism $f$ on $X$). These notions are important to guarantee the stability of a Borel measure at a point. Firstly, we recall necessary notions required in this section.

A point $x\in X$ is called an atom for a Borel measure $\mu$, if $\mu(\lbrace x\rbrace)>0$. 
A measure $\mu$ on $X$ is said to be non-atomic, if $\mu$ has no atom. 
Every Borel measure is assumed to be non-trivial i.e. $\mu(X) > 0$. 
Let $f$ be a homeomorphism on $X$. 
Then, a non-atomic Borel measure $\mu$ on $X$ is said to be expansive at $x\in X$ (with respect to $f$), if there exists   a $\mathfrak{c} > 0$ such that $\mu(\Phi_{f}^{\mathfrak{c}}(x))=0$, where $\Phi^{\mathfrak{c}}_{f}(x)=\lbrace y\in X\mid d(f^n(x),f^n(y))\leq \mathfrak{c}$, for each $n\in\mathbb{Z}\rbrace$ \cite{DKDM}. 
A Borel measure $\mu$ on $X$ is said to be expansive (with respect to $f$), if there exists a $\mathfrak{c} > 0$ such that $\mu(\Phi_{f}^{\mathfrak{c}}(x)) = 0$, for each $x\in X$. A homeomorphism $f$ is said to be a measure-expansive homeomorphism, if there exists a $\mathfrak{c} > 0$ (known as expansivity constant) such that for any non-atomic Borel measure $\mu$ on $X$ and for each $x\in X$, we have $\mu(\Phi_{f}^{\mathfrak{c}}(x)) = 0$ \cite{MM}.
Given a homeomorphism $h : X\rightarrow Y$ and a Borel measure $\mu$ on $X$, the pull-back measure on $Y$ is given by $h^{*}(\mu) = \mu \circ h^{-1}$.
\medskip

A Borel measure $\mu$ on $X$ has shadowing (with respect to $f$), if for every $\epsilon > 0$, there exists a $\delta > 0$ and a Borelian $B\subset X$ with $\mu(X\setminus B) = 0$ such that every $\delta$-pseudo orbit for $f$ through $B$ can be $\epsilon$-traced through $f$ by some point of $X$ \cite{LMT}.
\medskip

Let $Z$ be a subset of $X$. The set $2^{X}$ denotes the collection of all subsets of $X$. A map $H: Z \rightarrow 2^{X}$ denotes a set valued map. The set $Dom(H) = \lbrace z\in Z : H(z)\neq \phi\rbrace$ denotes the domain of $H$. A map $H$ is said to be compact valued if $H(z)$ is compact, for each $z\in Z$. For a given $\epsilon > 0$, we write $d^{X}(H, Id) < \epsilon$ whenever $H(z)\subset B[z, \epsilon]$, for each $z\in Z$. A map $H$ is said to be upper semi-continuous, if for each $z\in Dom(H)$ and for every neighborhood $O$ of $H(z)$, there exists a $\gamma > 0$ such that $H(x)\subset O$, whenever $d^{X}(x, z) < \gamma$, for each $x\in Z$.
\medskip

Let $f$ be a homeomorphism on $X$ and let $\mu$ be a Borel measure on $X$. Then, $\mu$ is said to be topologically stable (with respect to $f$), if for every $\epsilon > 0$, there exists a $\delta > 0$ such that for every homeomorphism $g$ satisfying $d_{C^{0}}(f, g) \leq \delta$, there exists an upper semi-continuous compact valued map $H: X\rightarrow 2^{X}$ with measurable domain such that:
\begin{enumerate}
\item[(i)] $\mu(X\setminus Dom(H)) = 0$.
\item[(ii)] $\mu (H(x)) = 0$, for each $x\in X$.
\item[(iii)] $d(H, Id) \leq \epsilon$.
\item[(iv)] $f\circ H = H\circ g$ \cite{LMT}.
\end{enumerate}

\begin{Definition}
Let $f$ be a homeomorphism on $X$, $\mu$ be a Borel measure on $X$ and let $x\in X$. Then,
\begin{enumerate}
\item $x$ is said to be a $\mu$-uniformly expansive point of $f$, if there exists a $\mathfrak{c} > 0$ such that $\mu(\Gamma_{f}^{\mathfrak{c}}(z)) = 0$, for each $z\in B(x, \mathfrak{c})$, where $\Gamma_{f}^{\mathfrak{c}}(z) = \lbrace y\in B(x, \mathfrak{c}) : d(f^{n}(y), f^{n}(z))\leq \mathfrak{c}\text{, for each }\hspace*{0.1cm} n\in \mathbb{Z}  \rbrace$. Such a constant $\mathfrak{c}$ is said to be an expansivity constant for $\mu$-uniform expansivity of $f$ at $x$.
\item $x$ is said to be a uniformly measure expansive point of $f$, if there exists a $\mathfrak{c} > 0$ such that for every non-atomic Borel measure $\mu$ on $X$ we have $\mu(\Gamma_{f}^{\mathfrak{c}}(z)) = 0$, for each $z\in B_{\mathfrak{c}}(x)$. Such a constant $\mathfrak{c}$ is said to be a expansivity constant for uniform measure expansivity of $f$ at $x$.
\end{enumerate}
The set of all $\mu$-uniformly expansive (uniformly measure expansive) points of $f$ is denoted by $UE_{f}^{\mu}(X)$($UE_{f}^{M}(X)$).
\label{D4.1}
\end{Definition}

\begin{Definition}
Let $f$ be a homeomorphism on $X$ and let $\mu$ be a Borel measure on $X$. A point $x\in X$ is said to be a $\mu$-shadowable point of $f$ if for every $\epsilon > 0$, there exists a $\delta > 0$ and a Borelian $B\subset X$ with $\mu(X\setminus B) = 0$ such that every $\delta$-pseudo orbit for $f$ through $B\cap B(x, \delta)$ is $\epsilon$-traced through $f$ by some point of $X$. The set of all $\mu$-shadowable points of $f$ is denoted by $Sh_{f}^{\mu}(X)$.
\label{D4.2}
\end{Definition}

\begin{Remark}
Let $f$ be a homeomorphism on $X$ and let $\mu$ be a Borel measure on $X$. Then, the following statements are easy to check.
\begin{enumerate}
\item[(i)] If $x\in UE_{f}^{\mu}(X)$ with expansivity constant $\mathfrak{c}$, then every point $z\in B(x, \frac{\mathfrak{c}}{4})$ is a $\mu$-uniformly expansive point of $f$ with expansivity constant $\frac{\mathfrak{c}}{4}$.
\item[(ii)]  $UE_{f}(X)\subset UE_{f}^{M}(X)$.
\item[(iii)] From \cite[Lemma 2.1]{MS}, we observe that a point $x\in X$ is a shadowable point of a homeomorphism $f$ on $X$ if and only if for every $\epsilon > 0$, there exists a $\delta > 0$ such that every $\delta$-pseudo-orbit for $f$ through $B(x, \delta)$ can be $\epsilon$-traced through $f$ by some point of $X$. Therefore, every shadowable point of a  homeomorphism $f$ on $X$ is a $\mu$-shadowable point of $f$, for every Borel measure $\mu$ on $X$.
\end{enumerate}
\label{R4.3}
\end{Remark}

\begin{Proposition}
Let $f$ be a homeomorphism on $X$ and let $\mu$ be a Borel measure on $X$. Then, the following statements are true: 
\begin{enumerate}
\item[(i)] $\mu$ is expansive (with respect to $f$) if and only if $UE_{f}^{\mu}(X) = X$.
\item[(ii)] $f$ is measure-expansive if and only if $UE^{M}_{f}(X) = X$.
\item[(iii)] $\mu$ has shadowing (with respect to $f$) if and only if $Sh_{f}^{\mu}(X) = X$.
\end{enumerate}
\label{T4.4}
\end{Proposition}
\begin{proof}
Let $f$ be a homeomorphism on $X$ and let $\mu$ be a Borel measure on $X$.
\begin{enumerate}
\item[(i)] Clearly, if $\mu$ is expansive (with respect to $f$) with expansivity constant $\mathfrak{c}$, then $x\in X$ is a $\mu$-uniformly expansive point of $f$ with expansivity constant $\mathfrak{c}$ and hence $UE_{f}^{\mu}(X) = X$. Conversely, suppose that $x\in X$ is a $\mu$-uniformly expansive point of $f$ with expansivity constant $\mathfrak{c}_{x}$. Consider an open cover $\Lambda = \lbrace B(x, \frac{\mathfrak{c}_{x}}{4}) : x\in X\rbrace$. Since $X$ is compact, we can choose $x_{1}, x_{2}, . . ., x_{k}$ such that $\cup_{i = 1}^{k}B(x_{i}, \frac{\mathfrak{c}_{x_{i}}}{4}) = X$. Set $\gamma = \min \lbrace \frac{\mathfrak{c}_{x_{1}}}{4}, \frac{\mathfrak{c}_{x_{2}}}{4}, . . ., \frac{\mathfrak{c}_{x_{k}}}{4}, \alpha\rbrace $, where $\alpha$ is a Lebesgue number for the covering $\Lambda$. Let $y\in X$. Then $y\in B(x_{i}, \frac{\mathfrak{c_{x_{i}}}}{4})$, for some $1\leq i \leq k$. From Remark \ref{R4.3}(i), we have $\mu(\Phi_{f}^{\gamma}(y)) = 0$. Since $y$ is chosen arbitrarily, we get that $\mu$ is expansive (with respect to $f$) with expansivity constant $\gamma$.
\item[(ii)] Proof is similar to the proof of (i).
\item[(iii)] Forward implication follows from corresponding definitions. Conversely, suppose that  $Sh_{f}^{\mu}(X) = X$. For a given $\epsilon > 0$, choose a $\delta_{x} > 0$ and a Borelian $B_{x}\subset X$ with $\mu(X\setminus B_{x}) = 0$ by $\mu$-shadowing of $f$ at $x$, for each $x\in X$. Since $X$ is compact, we can choose $x_{1}, x_{2}, . . ., x_{k}$ such that $\cup_{i=1}^{k}B(x_{i}, \delta_{x_{i}}) = X$. Consider a Borelian $B = \cap_{i=1}^{k}B_{x_{i}}$ and $\delta = \min_{i=1}^{k}\delta_{x_{i}}$. Since $X\setminus B\subset \cup_{i=1}^{k}(X\setminus B_{x_{i}})$, we have $\mu(X\setminus B) = 0$. It is easy to check that, every $\delta$-pseudo orbit for $f$ through $B$ can be $\epsilon$-traced through $f$ by some point of $X$. Since $\epsilon$ is chosen arbitrarily, we get that $\mu$ has shadowing (with respect to $f$).
\end{enumerate}
\end{proof}

\begin{Proposition}
Let $f$ be a homeomorphism on $X$ and let $\mu$ be a Borel measure on $X$. If $h$ is a homeomorphism from $X$ to $Y$, then the following statements are true:
\begin{enumerate}
\item[(i)] $UE_{h\circ f \circ h^{-1}}^{h^{*}(\mu)}(Y) =h(UE_{f}^{\mu}(X))$.
\item[(ii)] $Sh_{h\circ f \circ h^{-1}}^{h^{*}(\mu)}(Y) = h(Sh_{f}^{\mu}(X))$.
\end{enumerate} 
\label{T4.5}
\end{Proposition}
\begin{proof}
Let $f$ be a homeomorphism on $X$, $\mu$ be a Borel measure on $X$ and let $h$ be a homeomorphism from $X$ to $Y$.
\begin{enumerate}
\item[(i)] Let $y\in h(UE_{f}^{\mu}(X))$ and choose $z\in UE_{f}^{\mu}(X)$ with expansivity constant $\mathfrak{c}$ such that $y = h(z)$. 
Choose a $\mathfrak{d} > 0$ such that $d^{Y}(y_{1}, y_{2}) < \mathfrak{d}$ implies $d^{X}(h^{-1}(y_{1}), h^{-1}(y_{2}))$ $ < \mathfrak{c}$, for all $y_{1}, y_{2}\in Y$. It is easy to check that, $y\in UE_{h\circ f \circ h^{-1}}^{h^{*}(\mu)}(Y)$ with expansive constant $\mathfrak{d}$ and hence $h(UE_{f}^{\mu}(X))\subset UE_{h\circ f \circ h^{-1}}^{h^{*}(\mu)}(Y)$. 
Replace $h$ by $h^{-1}$, $f$ by $h\circ f\circ h^{-1}$ and $\mu$ by $h^{*}(\mu)$ in the last inclusion to complete the proof.

\item[(ii)] Let $y\in h(Sh_{f}^{\mu}(X))$ and choose $z\in Sh_{f}^{\mu}(X)$ such that $y = h(z)$. For a given $\epsilon > 0$, choose an $\eta > 0$ such that $d^{X}(x_{1},x_{2}) < \eta$ implies $d^{Y}(h(x_{1}), h(x_{2})) < \epsilon$, for all $x_{1}, x_{2}\in X$. Choose a $\gamma > 0$ and a Borelian $B\subset X$ with $\mu(X\setminus B) = 0$ such that every $\gamma$-pseudo orbit for $f$ through $B\cap B(z, \gamma)$ can be $\eta$-traced through $f$ by some point of $X$. Choose a $\delta > 0$ such that $d^{Y}(y_{1}, y_{2}) < \delta$ implies $d^{X}(h^{-1}(y_{1}), h^{-1}(y_{2})) < \gamma$, for all $y_{1}, y_{2}\in Y$. It is easy to check that, $h^{*}(\mu)(Y\setminus h(B)) = 0$ and every $\delta$-pseudo orbit $\lbrace y_{n}\rbrace_{n\in \mathbb{Z}}$ for $h\circ f\circ h^{-1}$ through $h(B)\cap B(y, \delta)$ can $\epsilon$-traced through $h\circ f\circ h^{-1}$ and hence $y\in Sh_{h\circ f \circ h^{-1}}^{h^{*}(\mu)}(Y)$. Since $y$ is chosen arbitrarily, we get that $h(Sh_{f}^{\mu}(X))\subset Sh_{h\circ f \circ h^{-1}}^{h^{*}(\mu)}(Y)$. Replace $h$ by $h^{-1}$, $f$ by $h\circ f\circ h^{-1}$ and $\mu$ by $h^{*}(\mu)$ in the last inclusion to complete the proof.
\end{enumerate}
\end{proof}

\begin{theorem} 
Let $\mu$ be a Borel measure on $X$, $\lbrace f_{n}\rbrace_{n\in \mathbb{N}}$ be a sequence of homeomorphisms on $X$ and let $f$ be a homeomorphism on $X$ such that $f_{n}\xrightarrow{uc} f$ and $f_{n}^{-1}\xrightarrow{uc} f^{-1}$. Then, $x\in UE_{f}^{\mu}(X)$ if and only if there exists a $\delta > 0$ such that $\mu(\lbrace y\in B(x, \delta) : \cup_{m\geq 1} \cap_{n\geq m} E_{z}(f_{n}, y, \frac{\delta}{3}) = \phi\rbrace) = 0$, for each $z\in B(x, \delta)$.
\label{T4.6}
\end{theorem}
\begin{proof}
Let $x$ be a $\mu$-uniformly expansive point of $f$ with expansivity constant $\mathfrak{c}$. For each $z\in B(x, \mathfrak{c})$, we set $A_{z} = \lbrace y\in B(x, \mathfrak{c}) : \cup_{m\geq 1} \cap_{n\geq m} E_{z}(f_{n}, y, \frac{\mathfrak{c}}{3}) = \phi\rbrace$. It is easy to check that, $A_{z}\subset \Gamma_{f}^{\mathfrak{c}}(z)$. Since $z$ is chosen arbitrarily, $\mu(\lbrace y\in B(x, \mathfrak{c}) : \cup_{m\geq 1} \cap_{n\geq m} E_{z}(f_{n}, y, \frac{\mathfrak{c}}{3}) = \phi\rbrace) \leq \mu(\Gamma_{f}^{\mathfrak{c}}(z)) = 0$, for each $z\in B(x, \mathfrak{c})$.
Conversely, choose a $\delta > 0$ such that $\mu(C_{z}) = 0$, for each $z\in B(x, \delta)$ where $C_{z} = \lbrace y\in B(x, \delta) : \cup_{m\geq 1} \cap_{n\geq m} E_{z}(f_{n}, y, \frac{\delta}{3}) = \phi\rbrace = 0$. 
It is easy to check that, $\Gamma_{f}^{\frac{\delta}{9}}(z)\subset C_{z}$, for each $z\in B(x, \frac{\delta}{9})$.
Hence, $\mu(\Gamma_{f}^{\frac{\delta}{9}}(z))\leq \mu(C_{z}) = 0$, for each $z\in B(x, \frac{\delta}{9})$. Since $z$ is chosen arbitrarily, we have $x\in UE_{f}^{\mu}(X)$ with expansivity constant $\frac{\delta}{9}$. 
\end{proof}

\begin{Definition}
Let $f$ be a homeomorphism on $X$ and let $\mu$ be a Borel measure on $X$. Then, a point $x\in X$ is said to be a $\mu$-topologically stable point of $f$, if for every $\epsilon > 0$, there exists a $\delta > 0$  such that for every homeomorphism $g$ satisfying $d_{C^{0}}(f, g) \leq \delta$, there exists an upper semi-continuous compact valued map $H: \overline{\mathcal{O}_{g}(x)}\rightarrow 2^{X}$ with the measurable domain such that:
\begin{enumerate}
\item[(i)] $\mu(H(z)) = 0$, for each $z\in B(x, \frac{\delta}{4})\cap \overline{\mathcal{O}_{g}(x)}$.
\item[(ii)] $d(H, Id) \leq \epsilon$.
\item[(iii)] $f\circ H = H\circ g$.
\end{enumerate}
Additionally, $x$ is said to be a strong $\mu$-topologically stable point of $f$, if there exists a Borelian $B\subset X$ with $\mu(X\setminus B) = 0$ such that:
\begin{enumerate}
\item[(iv)] $\mu(X\setminus Dom(H))\leq \mu(X\setminus U)$, where $U = B\cap B(x, \delta)\cap \overline{\mathcal{O}_{g}(x)}$.
\end{enumerate}
The set of all $\mu$-topologically stable (strong $\mu$-topologically stable) points of $f$ is denoted by $Ts_{f}^{\mu}(X)$ ($Sts_{f}^{\mu}(X)$).
\label{D5.2}
\end{Definition}

\begin{Proposition}
Let $f$ be a homeomorphism on $X$ and let $\mu$ be a Borel measure on $X$. Then, the following statements are true:
\begin{enumerate}
\item[(i)] If $\mu$ is topologically stable (with respect to $f$), then $TS_{f}^{\mu}(X) = X$.
\item[(ii)] For every homeomorphism $h: X\rightarrow Y$, $Ts_{h\circ f\circ h^{-1}}^{h^{*}(\mu)}(Y)  = h(Ts_{f}^{\mu}(X))$ and $Sts_{h\circ f\circ h^{-1}}^{h^{*}(\mu)}(Y)  = h(Sts_{f}^{\mu}(X))$.
\item[(iii)] If $\mu$ is non-atomic and $x$ is a topologically stable point of $f$, then $x$ is a strong $\mu$-topologically stable point of $f$. 
\end{enumerate}
\label{T5.3}
\end{Proposition}
\begin{proof}
Let $f$ be a homeomorphism on $X$ and let $\mu$ be a Borel measure on $X$.
\begin{enumerate}
\item[(i)] Proof follows from corresponding definitions.
\item[(ii)] Let $h: X\rightarrow Y$ be a homeomorphism. Let $y\in h(Ts_{f}^{\mu}(X))$ and choose $z\in Ts_{f}^{\mu}(X)$ such that $y = h(z)$. We claim that, $y\in Ts_{h\circ f \circ h^{-1}}^{h^{*}(\mu)}(Y)$. For a given $\epsilon > 0$, choose an $\eta > 0$ such that $d^{X}(x_{1},x_{2}) \leq \eta$ implies $d^{Y}(h(x_{1}), h(x_{2})) \leq \epsilon$, for all $x_{1}, x_{2}\in X$. For this $\eta$, choose a $\gamma > 0$ (and a Borelian $B\subset X$) by $\mu$-topological stability (strong $\mu$-topological stability) of $z$. Choose a $\delta > 0$ such that $d^{Y}(y_{1}, y_{2}) \leq \delta$ implies $d^{X}(h^{-1}(y_{1}), h^{-1}(y_{2})) \leq \frac{\gamma}{4}$, for all $y_{1}, y_{2}\in Y$. Clearly, $h^{*}(\mu)(Y\setminus h(B)) = \mu(h^{-1}(h(X)\setminus h(B))) = \mu(h^{-1}(h(X\setminus B))) = 0$. If $g$ is a  homeomorphism on $Y$ satisfying $d_{C^{0}}(h\circ f\circ h^{-1}, g) \leq \delta$, then $h^{-1}\circ g\circ h$ is a homeomorphism on $X$ satisfying $d_{C^{0}}(f, h\circ g\circ h^{-1}) \leq \gamma$. Hence, there exists an upper semi-continuous compact valued map $H: \overline{\mathcal{O}_{h^{-1}\circ g\circ h}(z)}\rightarrow 2^{X}$ with the measurable domain satisfying $\mu(H(u)) = 0$, for each $u\in B(z, \frac{\gamma}{4})\cap \overline{\mathcal{O}_{h^{-1}\circ g\circ h}(z)}$, $d^{X}(H, Id) \leq \eta$ and $f\circ H = H\circ h^{-1}\circ g\circ h$. 
\medskip

Since $\overline{\mathcal{O}_{h^{-1}\circ g\circ h}(z)} = h^{-1}(\overline{\mathcal{O}_{g}(y)})$, we can define $K : \overline{\mathcal{O}_{g}(y)} \rightarrow 2^{Y}$ by $K(u) = h\circ H\circ h^{-1}(u)$, for each $u\in \overline{\mathcal{O}_{g}(y)}$. Clearly, $K$ is an upper semi-continuous compact-valued map. Since, $u\in Dom(K)$ if and only if $u\in h(Dom(H))$, $K$ has measurable domain. Since, $h^{-1}( B(y, \frac{\delta}{4})\cap \overline{\mathcal{O}_{g}(y)})\subset B(z, \frac{\gamma}{4})\cap \overline{\mathcal{O}_{h^{-1}\circ g\circ h}(z)}$, we have $h^{*}(\mu)(K(u)) = \mu(h^{-1}\circ h\circ H\circ h^{-1}(u)) = \mu(H\circ h^{-1}(u)) = 0$, for each $u\in B(y, \frac{\delta}{4})\cap \overline{\mathcal{O}_{g}(y)}$. Since $d^{X}(H, Id) \leq \eta$, we have $d^{Y}(K, Id) \leq \epsilon$. Also, $h\circ f\circ h^{-1}\circ K = h\circ f\circ h^{-1}\circ h\circ H\circ h^{-1} = h\circ f \circ H\circ h^{-1} =  h\circ H\circ h^{-1}\circ g\circ h\circ  h^{-1} = K\circ g$. Moreover, suppose that $\mu(X\setminus Dom(H))\leq \mu(X\setminus U)$, where $U = B\cap B(z, \gamma)\cap \overline{\mathcal{O}_{h^{-1}\circ g\circ h}(z)}$. Set $V = h(B)\cap B(y, \delta)\cap \overline{\mathcal{O}_{g}(y)}$. Clearly, $h^{-1}(V)\subset U$. Therefore, $h^{*}(\mu)(Y\setminus Dom(K)) \leq h^{*}(\mu)(Y\setminus V)$. Hence, $ h(Ts_{f}^{\mu}(X))\subset Ts_{h\circ f\circ h^{-1}}^{h^{*}(\mu)}(Y)$ ($h(Sts_{f}^{\mu}(X))\subset Sts_{h\circ f\circ h^{-1}}^{h^{*}(\mu)}(Y)$). Replace $h$ by $h^{-1}$, $f$ by $h\circ f\circ h^{-1}$ and $\mu$ by $h^{*}(\mu)$ in the last inclusion to complete the proof.

\item[(iii)] Let $\mu$ be a non-atomic measure and let $x$ be a topologically stable point of $f$. For a given $\epsilon > 0$, choose a $\delta > 0$ by the topological stability of $x$ and fix $B = X$. Let $g$ be a homeomorphism satisfying $d_{C^{0}}(f, g) \leq \delta$. By the topological stability of $x$, there exists a continuous map $h : \overline{\mathcal{O}_{g}(x)}\rightarrow X$ satisfying $f\circ h = h\circ g$ and $d(h(y), y) \leq \epsilon$, for each $y\in \overline{\mathcal{O}_{g}(x)}$. Define a set valued map $H: \overline{\mathcal{O}_{g}(x)}\rightarrow 2^{X}$ by $H(y) = \lbrace h(y)\rbrace$, for each $y\in \overline{\mathcal{O}_{g}(x)}$. It is easy to check that, $H$ is an upper semi-continuous compact valued map with the measurable domain and satisfies Definition \ref{D5.2}(i)-(iv). Hence, $x$ is a strong $\mu$-topologically stable point of $f$.
\end{enumerate}
\end{proof}

\begin{theorem}
Let $f$ be a homeomorphism on $X$ and let $\mu$ be a Borel measure on $X$. If $x\in X$ is a $\mu$-uniformly expansive point of $f$, then $x$ is a $\mu$-topologically stable point of $f$. Moreover, if $x$ is also a $\mu$-shadowable point of $f$, then $x$ is a strong $\mu$-topologically stable point of $f$.
\label{C5.10}
\end{theorem}
\begin{proof}
We prove the later case. Former case follows by choosing $\delta = \eta$ in the proof.
Let $f$ be a homeomorphism on $X$ and let $\mu$ be a Borel measure on $X$. Let $x\in X$ be a $\mu$-uniformly expansive $\mu$-shadowable point of $f$ with expansivity constant $\mathfrak{c}$. For a given $\epsilon > 0$, fix $0 < \eta < \min\lbrace \frac{\mathfrak{c}}{8}, \epsilon\rbrace$. 
For this $\eta$, choose a $\delta > 0$ and a Borelian $B$ by shadowing at $x$. Choose a homeomorphism $g$ satisfying $d_{C^{0}}(f, g) \leq \delta$. Define a set-valued map $H: \overline{\mathcal{O}_{g}(x)}\rightarrow 2^{X}$ by $H(u) = \cap_{n\in \mathbb{Z}} f^{-n}(B_{\eta}[g^{n}(u)])$, for each $u\in \overline{\mathcal{O}_{g}(x)}$. Clearly, $H$ is a compact valued map.
We claim that, $Dom(H)$ is measurable. Choose a sequence $x_{k}\in Dom(H)$ converging to $z\in X$. Since $\overline{\mathcal{O}_{g}(x)}$ is closed, $z\in \overline{\mathcal{O}_{g}(x)}$. Choose $y_{k}\in X$ such that $d(f^{n}(y_{k}), g^{n}(x_{k})) \leq \eta$, for each $k\in \mathbb{N}$ and for each $n\in \mathbb{Z}$. Since $X$ is compact, we can assume that $y_{k}$ converges to $y\in X$. But then $y\in H(z)$ and hence $z\in Dom(H)$. Therefore, $Dom(H)$ is closed and hence measurable.
\medskip

Choose $w\in V = B(x, \frac{\delta}{4})\cap \overline{\mathcal{O}_{g}(x)}$ and fix $y\in H(w)$. If $z\in H(w)$, we have $d(f^{n}(z), g^{n}(w))\leq \eta$, for each $n\in \mathbb{Z}$. Since $y\in H(w)$, we have $d(f^{n}(y), f^{n}(z) \leq d(f^{n}(y),g^{n}(w)) +  d(g^{n}(w), f^{n}(z)) \leq 2\eta < \frac{\mathfrak{c}}{4}$. Therefore, $H(w)\subset \Gamma_{f}^{\frac{\mathfrak{c}}{4}}(y)$. Since, $y\in B(w, \frac{\mathfrak{c}}{4})$ and $w\in B(x, \frac{\mathfrak{c}}{4})$, from Remark \ref{R4.3}(i) we have $\mu(H(w))\leq \mu(\Gamma_{f}^{\frac{\mathfrak{c}}{4}}(y)) = 0$. Since, $w$ is chosen arbitrarily, $\mu(H(w)) = 0$, for each $w\in V$. Moreover $H(z)\subset B[z, \eta]$, for each $z\in \overline{\mathcal{O}_{g}}(x)$ and hence $d(H, Id)\leq \epsilon$. 
Now, we prove that $f\circ H = H\circ g$. If $z\in Dom(H)$, then $H(z)\neq \phi$ and hence $f(H(z)) = f(\cap_{n\in \mathbb{Z}} f^{-n}(B[g^{n}(z), \eta])) = \cap_{n\in \mathbb{Z}} f^{-n+1}(B[g^{n}(z), \eta]) = \cap_{n\in \mathbb{Z}} f^{-n}(B[g^{n}(g(z)), \eta]) = H(g(z))$. Thus $f\circ H = H\circ g$ in $Dom(H)$. If $z\notin Dom(H)$, then $g(z)\notin Dom(H)$ and hence $f(H(z)) = \phi = H(g(z))$. Therefore,  $f\circ H = H\circ g$ in $\overline{\mathcal{O}_{g}(x)}\setminus Dom(H)$ also. 
\medskip

To prove that $H$ is upper semi-continuous, choose a $w\in Dom(H)$ and a neighborhood $O$ of $H(w)$. Set $H(y) = \cap_{i=0}^{\infty} H_{i}(y)$ and $H_{i}(y) = \cap_{n=-i}^{i}f^{-n}(B[g^{n}(y), \eta])$. Clearly, each $H_{i}(y)$ is compact and $H_{i+1}(y)\subset H_{i}(y)$, for every $i\geq 0$ and for each $y\in \overline{\mathcal{O}_{g}(x)}$. For $y = w$, we obtain $i\geq 0$ such that $H_{i}(w)\subset O$. We claim that, there exists a $\gamma > 0$ such that $H_{i}(y)\subset O$ whenever $d(w, y) < \gamma$, for each $y\in \overline{\mathcal{O}_{g}(x)}$. On the contrary, choose a sequence $y_{k}\rightarrow w$ in  $\overline{\mathcal{O}_{g}(x)}$ and $z_{k}\in H_{i}(y_{k})\setminus O$, for each $k\in \mathbb{N}$. Since $X$ is compact, we can assume that $z_{k}\rightarrow z$, for some $z\in X\setminus O$. Since $z_{k}\in H_{i}(y_{k})$, we have $d(f^{n}(z_{k}), g^{n}(y_{k}))\leq \eta$, for each $k\in \mathbb{N}$ and for all $-i\leq n\leq i$. Therefore, $d(f^{n}(z), g^{n}(w))\leq \eta$, for all $-i\leq n\leq i$ and hence $z\in H_{i}(w)$. Since $z\notin O$ and $H_{i}(w)\subset O$, we get a contradiction. Since $H(y)\subset H_{i}(y)$, there exists a $\gamma > 0$ such that $H(y)\subset O$ whenever $d(w, y) < \gamma$, for each $y\in \overline{\mathcal{O}_{g}(x)}$. 
\medskip

Now, note that $d_{C^{0}}(f,g) \leq \delta$ implies that $\lbrace g^{n}(z)\rbrace_{n\in \mathbb{Z}}$ is a $\delta$-pseudo orbit for $f$ through $z$, for each $z\in U = B\cap B(x, \delta)\cap \overline{\mathcal{O}_{g}(x)}$. Choose $y\in X$ such that $d(f^{n}(y), g^{n}(z))\leq \eta$, for each $n\in \mathbb{Z}$. Therefore, $H(z)\neq \phi$, for each $z\in U$ and hence $ U\subset Dom(H)$ implying $\mu(X\setminus Dom(H))\leq \mu(X\setminus U)$.
\end{proof}

\begin{Corollary}
Let $f$ be a homeomorphism on $X$ and let $\mu$ be a Borel measure on $X$. If $x\in X$ is a uniformly expansive point of $f$, then $x$ is a $\mu$-topologically stable point of $f$. Moreover, if $x$ is also a shadowable point of $f$, then $x$ is a strong $\mu$-topologically stable point of $f$.
\label{C5.11}
\end{Corollary}

\textbf{Acknowledgements:} The first author is supported by CSIR-Junior Research Fellowship (File No.-09/045(1558)/2018-EMR-I) of  Government of India.

\end{document}